\documentclass[10pt]{amsart}
\usepackage[utf8]{inputenc}

\usepackage{amssymb,amsthm,amsmath}
\usepackage{mathrsfs}
\usepackage{enumerate}
\usepackage{graphicx,xcolor}
\usepackage[hidelinks]{hyperref}

\newcommand{\ls}{\leq}
\newcommand{\gr}{\geq}

\newcommand{\dd}{\mathrm{d}}
\newcommand{\E}{\mathbb{E}}
\newcommand{\pp}{\mathbb{P}}
\newcommand{\1}{\textbf{1}}
\newcommand{\R}{\mathbb{R}}

\newcommand{\red}{}

\DeclareMathOperator{\Var}{Var}

\usepackage{xpatch}
\makeatletter
\def\thm@space@setup{%
  \thm@preskip=12pt plus 0pt minus 0pt
  \thm@postskip=0pt plus 0pt minus 0pt
}
\xpatchcmd{\proof}{6\p@\@plus6\p@\relax}{\z@skip}{}{}
\makeatother

\newtheorem{theorem}{Theorem}
\newtheorem{lemma}[theorem]{Lemma}
\newtheorem{corollary}[theorem]{Corollary}

\theoremstyle{remark}
\newtheorem{remark}[theorem]{Remark}

\theoremstyle{definition}

\title{Entropies of sums of independent gamma random variables}

\author{Giorgos Chasapis*}\thanks{*Corresponding author. Email: \texttt{gchasapi@andrew.cmu.edu}}

\author{Salil Singh}

\author{Tomasz Tkocz}

\address{Carnegie Mellon University; Pittsburgh, PA 15213, USA.}

\email{\{gchasapi,salils,ttkocz\}@andrew.cmu.edu}

\thanks{TT's research supported in part by NSF grant DMS-1955175.}

\date{\today}

\begin{document}

\begin{abstract} 
We establish several Schur-convexity type results under fixed variance for weighted sums of independent gamma random variables and obtain nonasymptotic bounds on their R\'enyi entropies. In particular, this pertains to the recent results by Bartczak-Nayar-Zwara as well as Bobkov-Naumov-Ulyanov, offering simple proofs of the former and extending the latter.
\end{abstract}

\maketitle

\bigskip

\begin{footnotesize}
\noindent {\em 2010 Mathematics Subject Classification.} Primary 60E15; Secondary 94A17.

\noindent {\em Key words. Entropy, max-entropy, gamma distribution, weighted sums, Schur-convexity} 
\end{footnotesize}

\bigskip

\section{Introduction}

Suppose $X_1, X_2, \dots, X_n$ are independent, identically distributed (i.i.d.) square-integrable random variables, say with variance $1$ and $a = (a_1, a_2, \dots, a_n)$ is a unit vector in $\R^n$, $\sum_{k=1}^n a_k^2 = 1$, so that the variance of the sum $X_a = \sum_{k=1}^n a_kX_k$ does not depend on $a$ and also equals $1$. For which vectors $a$, is the distribution of $X_a$ as close to the Gaussian distribution as possible? A natural way to quantify this vague question is to measure the distance to Gaussianity via relative entropy and ask about $\inf_a D(X_a||G)$. Here $D(X||G) = h(G) - h(X)$ is the relative entropy of $X$ with respect to a Gaussian random variable $G$ of the same variance as $X$, where 
\[
h(X) = -\int_{\R} f\log f
\]
is the Shannon entropy of a random variable $X$ with density $f$. 

This question was raised in \cite{ENT1} and addressed for (symmetric) \emph{Gaussian mixtures}, where the extremising sequence turns out to be simply $a = (\frac{1}{\sqrt{n}}, \dots, \frac{1}{\sqrt{n}})$. In a recent paper \cite{BNZ}, Bartczak, Nayar and Zwara considered the case of gamma distribution and established that the same vector is extremal among all nonnegative vectors, that is whose all components are nonnegative. We refer to their paper for a comprehensive account of relevant reference and related problems. Their approach rests on the so-called method of interlacing densities (see also \cite{ENT2}). For the gamma distribution, this entails a rather technical and involved analysis for Bessel functions.

Our first goal in this paper is to offer an alternative approach. It turns out that for the gamma distribution, simple arguments involving moment generating functions allow to establish certain Schur-convexity type results. Those in particular give the main result of \cite{BNZ}, as well as partially address Question 6 from \cite{BNZ} about moments.

Our second goal in this paper is to extend a recent result of \cite{BNU}, where Bobkov, Naumov and Ulyanov find a nonasymptotic expression for the maximum of the density of $X_a = \sum a_k X_k$ with $X_k$ having $\Gamma(1/2)$ distribution, equivalently for the $\infty$-R\'enyi entropy of $X_a$ in terms of $a$. We extend this to $\Gamma(\gamma)$ distribution with $\gamma \geq 1/2$. Such bounds have applications to L\'evy's concentration function, thus to anti-concentration inequalities (see, e.g. \cite{BCh} as well as e.g. the survey \cite{HV} for an exposition on anti-concentration).

Another piece of motivation to study such extensions is the fact that weighted sums of independent $\Gamma(1/2)$ random variables emerge naturally from Gaussian quadratic forms, which was a starting point for both \cite{BNZ} and \cite{BNU}.

In the next section, we recall the definition of R\'enyi entropy and formulate our results. The remaining part of this note will be devoted to their proofs.

\subsection*{Acknowledgements.} We are indebted to Han Nguyen for many fruitful and illuminating discussions. We should like to thank an anonymous referee very much for their insightful comments greatly improving this manuscript.

\section{Results}

Let $0 \leq \alpha \leq \infty$. For a random variable $X$ with density $f$, we define its R\'enyi entropy of order $\alpha$ as
\[
h_\alpha(X) = \frac{1}{1-\alpha}\log \int f^\alpha,
\]
see \cite{R}, understood as limits in the cases $\alpha \in \{0,1,\infty\}$, namely $h_0(X) = \log |\text{supp}(f)|$, $h_1(X) = h(X)$ (the Shannon entropy), $h_\infty(X) = -\log \|f\|_\infty$. For notational convenience, we also introduce the functional
\[
M(X) = \|f\|_\infty.
\]

Throughout, we let $\gamma > 0$ and let $X_1, X_2, \dots$ be i.i.d. random variables with $\Gamma(\gamma)$ distribution, that is with density $\Gamma(\gamma)^{-1}x^{\gamma-1}e^{-x}\1_{(0,\infty)}(x)$ on $\R$.

\subsection{Schur-convexity type results and entropies}

Our first main result gives the Schur-concavity of \emph{centred} weighted sums averaged against  arbitrary completely monotone functions. For concise exposition on majorisation and Schur-convexity, we refer for instance to Chapter II of \cite{Bh}. We recall that a function $\Phi\colon (0,+\infty) \to (0,+\infty)$ is completely monotone if it is a \emph{mixture} of exponential functions, that is $\Phi(x) = \int_{0}^\infty e^{-tx}\dd\mu(t)$ for some nonnegative Borel measure $\mu$, equivalently (by Bernstein's theorem) $(-1)^{m}\Phi^{(m)}(x) \geq 0$ for every $m = 0, 1, 2, \dots$, see, e.g. \cite{Fel}.

\begin{theorem}\label{thm:Phi}
For a completely monotone function $\Phi\colon (0,+\infty) \to (0,+\infty)$ and $c > 0$, the function
\begin{equation}\label{eq:Phi}
(a_1, \dots, a_n) \mapsto \E\Phi\left(c + \sum_{j=1}^n \sqrt{a_j}(X_j-\gamma)\right)
\end{equation}
is Schur-concave on the simplex $\{a \in \R_+^n, \ \sum a_j < \frac{c^2}{\gamma^2n}\}$.
\end{theorem}

We emphasise that the centering of the $X_j$ by its mean $\E X_j = \gamma$ is crucial for this result to hold. Without the centering, the resulting function is Schur-convex, as will follow from our proof. 

\begin{theorem}\label{thm:Phi0}
For a completely monotone function $\Phi\colon (0,+\infty) \to (0,+\infty)$, the function
\begin{equation}\label{eq:Phi0}
(a_1, \dots, a_n) \mapsto \E\Phi\left(\sum_{j=1}^n \sqrt{a_j}X_j\right)
\end{equation}
is Schur-convex on $\R_+^n$.
\end{theorem}

The main result of \cite{BNZ} follows as a corollary to Theorem \ref{thm:Phi}.

\begin{corollary}[Bartczak-Nayar-Zwara, \cite{BNZ}]\label{cor:log}
Provided that $\gamma n \geq 1$, we have for the Shannon entropy,
\begin{equation}\label{eq:ent}
h\left(\sum_{j=1}^n \sqrt{a_j}X_j\right) \leq h\left(\sum_{j=1}^n \frac{1}{\sqrt{n}}X_j\right),
\end{equation}
whenever $\sum a_j = 1$.
\end{corollary}

For general $\alpha$-R\'enyi entropies, we can deduce the same, but using Theorem \ref{thm:Phi0} and imposing additional restrictions on the parameters $\alpha, \gamma$ and $n$. This can be compared with results for Gaussian mixtures (Theorem 8 in \cite{ENT1}) as well as sums of uniform random variables (Theorem 2 in \cite{CGT}).

\begin{corollary}\label{cor:renyi}
Let $\alpha > 1$, $n\gamma < 1$ and $\sum_{j=1}^n a_j = 1$. We have,
\begin{equation}\label{eq:renyi}
h_\alpha\left(\sum_{j=1}^n \sqrt{a_j}X_j\right) \leq h_\alpha\left(\sum_{j=1}^n \frac{1}{\sqrt{n}}X_j\right).
\end{equation}
\end{corollary}

We also have Schur-convexity for power functions with integral exponents. This relates to Question 6 from \cite{BNZ}, except that here we are only able to handle \emph{centred} moments of even order.

\begin{theorem}\label{thm:Sk}
For every positive integer $k$, the function
\begin{equation}\label{eq:Sk}
(a_1, \dots, a_n) \mapsto \E\left(\sum_{j=1}^n \sqrt{a_j}(X_j-\gamma)\right)^k
\end{equation}
is Schur-convex on $\R_+^n$.
\end{theorem}

\subsection{Maximum density}

Let $a_1 \geq \dots \geq a_n > 0$, $\sum_{j=1}^n a_j = 1$. The main result of Bobkov-Naumov-Ulyanov from \cite{BNU} asserts that when $\gamma = 1/2$ (i.e. when $X_j$ has the same distribution as $\frac{1}{2}Z^2$), we have
\begin{equation}\label{eq:BNU}
\frac{1}{2e^2\sqrt{2\pi}}(1-a_1)^{-1/4} \leq M\left(\sum_{j=1}^n \sqrt{a_j}X_j\right) \leq \frac{4}{\sqrt{\pi}}(1-a_1)^{-1/4},
\end{equation}
Using their approach, we extend this to $\gamma \geq 1/2$. For $\gamma \geq 1$, $M$ is of the constant order, $\gamma^{-1/2}$ up to universal constants. For $\frac{1}{2} \leq \gamma < 1$, only the exponent in \eqref{eq:BNU} has to be modified (and of course the universal constants). 

\begin{theorem}\label{thm:M}
For $\gamma \geq 1$, we have
\begin{equation}\label{eq:g1}
\frac{1}{\sqrt{12}}\gamma^{-1/2} \leq M\left(\sum_{j=1}^n \sqrt{a_j}X_j\right) \leq \gamma^{-1/2}.
\end{equation}
For $\frac{1}{2} \leq \gamma < 1$, there are constants $c_\gamma$ and $C_\gamma$ for which we have
\begin{equation}\label{eq:g12}
c_\gamma(1-a_1)^{\frac{\gamma-1}{2}} \leq M\left(\sum_{j=1}^n \sqrt{a_j}X_j\right)  \leq C_\gamma(1-a_1)^{\frac{\gamma-1}{2}}.
\end{equation}
The lower bound in fact holds for every $0 < \gamma < 1$ and we can take $c_\gamma = 0.003\gamma$.
\end{theorem}

{\red Bounds \eqref{eq:g1} do not use the particular structure of the sum $\sum \sqrt{a_j}X_j$: the lower bound holds for \emph{all} random variables, whereas the upper bound holds for \emph{all log-concave} random variables (see Lemmas \ref{lm:Moriguti} and \ref{lm:BCh} below) with the constant $1$ in front of $\gamma^{-1/2}$ being in fact optimal (attained for the one-sided exponential distribution, so when $n=1$ and $\gamma=1$).}
For $\gamma < 1$, we will prove slightly more general results, allowing to justify the following remark.

\begin{remark}\label{rem:M}
For $\gamma < 1$, $M = +\infty$ regardless of $a$ as long as $n \leq \lfloor 1/\gamma \rfloor$ (e.g. see Lemma \ref{lm:density} below). For $\gamma < \frac{1}{2}$, we only know the matching lower and upper bounds when $n = \lfloor 1/\gamma \rfloor + 1$ (see Remark \ref{rem:smalln} in the next section). 
The case of arbitrary $n$ has been elusive and we find it an interesting question.
\end{remark}

{\red
One final comment is in place: all of our results can be naturally interpreted in information theoretic language as R\'enyi entropy bounds. There has been significant amount of work devoted to such bounds, see for instance \cite{BN, BCh2, BMM, BM, Li, LMM, MMX, MNT, MM, MT, RS} for recent results and additional references.
In particular, inequality \eqref{eq:g12} can be recast as a comparison between the (contiuous) $\infty$-R\'{e}nyi entropy of $\sum_{j=1}^n \sqrt{a_j}X_j$ and the (discrete) $\infty$-R\'enyi entropy of the coefficient vector $a=(a_1,\ldots,a_n)$: up to multiplicative constants, we have
\[
1-e^{-h_\infty(a)} \approx_\gamma e^{\frac{2}{1-\gamma}h_\infty\left(\sum_{j=1}^n \sqrt{a_j}X_j\right)}.
\]
}

\section{Proofs}

We note for future use the formula for the moment generating function of a $\Gamma(\gamma)$ random variable $X$: for $t < 1$,
\begin{equation}\label{eq:mgf}
\E e^{t X} = (1-t)^{-\gamma}.
\end{equation}

\subsection{Proof of Theorems \ref{thm:Phi} and \ref{thm:Phi0}}

We begin with a lemma.

\begin{lemma}\label{lm:F}
Let $\gamma > 0$. The function $F(x_1,\dots,x_n) = \prod_{j=1}^n e^{\gamma\sqrt{x_j}}(1+\sqrt{x_j})^{-\gamma}$ is Schur-concave on $\R_+^n$, whereas the function $G(x_1,\dots,x_n) = \prod_{j=1}^n (1+\sqrt{x_j})^{-\gamma}$ is Schur-convex on $\R_+^n$.
\end{lemma}
\begin{proof}
For function $F$, we have,
\begin{align*}
\frac{1}{F(x)}\frac{\partial F}{\partial x_k} = \frac{\partial}{\partial x_k}\log F &= \gamma\frac{\partial}{\partial x_k}\left(\sqrt{x_k}-\log(1+\sqrt{x_k})\right) \\
&= \frac{\gamma}{2\sqrt{x_k}}\left(1 - \frac{1}{1+\sqrt{x_k}}\right) \\
&= \frac{\gamma}{2}\frac{1}{1+\sqrt{x_k}}.
\end{align*}
Thus, if $x_k > x_l$, then
\[
\frac{\partial F}{\partial x_k} - \frac{\partial F}{\partial x_l} = \frac{\gamma}{2} F(x)\left(\frac{1}{1+\sqrt{x_k}} - \frac{1}{1+\sqrt{x_l}}\right) < 0.
\]
The Schur-Ostrowski criterion finishes the proof for $F$. For function $G$, the argument proceeds identically.
\end{proof}

\begin{proof}[Proof of Theorem \ref{thm:Phi}]
Since $\Phi$ is completely monotone, there is a (nonnegative) Borel measure on $\mu$ such that
\[
\Phi(x) = \int_0^\infty e^{-t x} \dd \mu(t).
\]
Thus, thanks to independence and \eqref{eq:mgf},
\[
\E\Phi\left(c + \sum_{j=1}^n \sqrt{a_j}(X_j-\gamma)\right) = \int_{0}^\infty \left(\prod_{j=1}^n e^{\gamma t\sqrt{a_j}}(1+t\sqrt{a_j})^{-\gamma}\right) e^{-ct}\dd\mu(t).
\]
Lemma \ref{lm:F} finishes the proof.
\end{proof}

\begin{remark}
We emphasise that the factor $e^{\gamma \sqrt{a_j}}$ appears as a result of centreing the $X_j$. This factor is crucial for function $F$ from Lemma \ref{lm:F} to be Schur-concave, as without it, as we have seen, it is Schur-convex. Theorem \ref{thm:Phi0} follows analogously.
\end{remark}

\subsection{Proof of Corollary \ref{cor:log}}

First note that applying Theorem \ref{thm:Phi} to $\Phi(x) = x^{-q}$ with $q \to 0+$ and using that $\frac{x^{-q}-1}{q} \downarrow -\log x$, as $q \downarrow 0+$ for positive $x$, we conclude that Theorem \ref{thm:Phi} also holds with $\Phi(x) = -\log x$. To prove \eqref{eq:ent}, fix $c > \gamma\sqrt{n}$ and positive $a_j$ with $\sum_{j=1}^n a_j = 1$. Recall that for an arbitrary probability density function $g$, we have
\begin{align*}
h\left(\sum_{j=1}^n \sqrt{a_j}X_j\right) &= h\left(c+\sum_{j=1}^n \sqrt{a_j}(X_j-\gamma)\right) \\
&\leq \E \left[-\log g\left(c+ \sum_{j=1}^n \sqrt{a_j}(X_j-\gamma)\right)\right].
\end{align*}
Letting $g$ be the density of $\sum_{j=1}^n \frac{1}{\sqrt{n}}X_j$, that is
\[
g(x) = \frac{n^{(\gamma n-1)/2}}{\Gamma(\gamma n)}x_+^{\gamma n - 1}e^{-x\sqrt{n}},
\]
we thus obtain from the first part that (note that we need $\gamma n - 1 \geq 0$)
\begin{align*}
h\left(\sum_{j=1}^n \sqrt{a_j}X_j\right) &\leq \E \left[-\log g\left(c+ \sum_{j=1}^n \frac{1}{\sqrt{n}}(X_j-\gamma)\right)\right]
\end{align*}
With $c \to \gamma\sqrt{n}+$, the right hand side becomes $h\left(\sum_{j=1}^n \frac{1}{\sqrt{n}}X_j\right)$.\hfill$\square$

\subsection{Proof of Corollary \ref{cor:renyi}}

Suppose $\sum a_j = 1$, let $f$ be the density of $\sum \sqrt{a_j}X_j$ and $g$ be the density of $\sum X_j/\sqrt{n}$. Our goal is to show that $\int f^\alpha \geq \int g^\alpha$. By H\"older's inequality,
\[
\left(\int f^\alpha\right)^{\frac{1}{\alpha}}\left(\int g^\alpha\right)^{\frac{\alpha-1}{\alpha}} \geq \int fg^{\alpha-1}.
\]
Note that the right hand side reads $\E \Phi(\sum \sqrt{a_j}X_j)$ with 
\[\Phi(x) = g(x)^{\alpha-1} = \Big(\Gamma(n\gamma)^{-1}(x\sqrt{n})^{n\gamma-1}e^{-x\sqrt{n}}\Big)^{\alpha-1}
\]
which is completely monotone as a product of two completely monotone functions (hence the assumptions, to have $(n\gamma-1)(\alpha-1) <0$ and $\alpha-1>0$). It remains to apply Theorem \ref{thm:Phi0} to the sequence $(a_j)$ which always majorises the constant sequence $(\frac{1}{n})$.
\hfill$\square$

{\red
\begin{remark}
The application of H\"{o}lder's inequality in the form of a variational formula for the Renyi entropy, as in the proof of Corollary \ref{cor:renyi}, has been recently used in a number of information theoretic contexts and can be probably traced back to \cite{Yu}.
\end{remark}
}

\subsection{Proof of Theorem \ref{thm:Sk}}

First we prove a lemma about centred integral moments of a single summand.

\begin{lemma}\label{lm:mom}
For every positive integer $k$, $\E(X_1-\gamma)^k \geq 0$.
\end{lemma}
\begin{proof}
Rephrasing the lemma, it suffices to prove that the power-series expansion of the moment generating function $\E e^{t(X_1-\gamma)}$ has nonnegative coefficients. Invoking \eqref{eq:mgf} and using that $-\log(1-t) = \sum_{k=1}^\infty \frac{t^k}{k}$, we obtain
\[
\E e^{t(X_1-\gamma)} = \exp\left\{\gamma\left(-t - \log(1-t)\right)\right\} =  \exp\left\{\gamma\sum_{k=2}^\infty \frac{t^k}{k}\right\}.
\]
Since the power series expansion of $\exp$ has positive coefficients, the proof is complete.
\end{proof}

\begin{proof}[Proof of Theorem \ref{thm:Sk}]
Let $S = \sum_{j=1}^n \sqrt{a_j}(X_j-\gamma)$. Consider for sufficiently small positive $t$,
\begin{align*}
\E e^{tS} &= \sum_{k=0}^\infty \frac{t^k}{k!}\E S^k \\
&= \prod_{j=1}^n \exp\left\{-\gamma\Big(t\sqrt{a_j} + \log(1-t\sqrt{a_j})\Big)\right\}.
\end{align*}
Call the right hand side $F$. Fix two indices $i \neq j$ and observe that $\left(\frac{\partial}{\partial a_j} - \frac{\partial}{\partial a_i}\right)\E S^k$ is the Taylor coefficient of $\left(\frac{\partial}{\partial a_j} - \frac{\partial}{\partial a_i}\right)F$ at $t^k$. On the other hand,
\[
\frac{\partial F}{\partial a_j} = F\cdot \left(-\gamma\left(\frac{t}{2\sqrt{a_j}} - \frac{t}{2\sqrt{a_j}(1-t\sqrt{a_j})}\right)\right) = F\cdot\frac{\gamma t^2}{2(1-t\sqrt{a_j})}.
\]
Thus,
\[
\left(\frac{\partial}{\partial a_j} - \frac{\partial}{\partial a_i}\right)F = F\cdot \frac{\gamma t^3}{2(1-t\sqrt{a_j})(1-t\sqrt{a_i})}(\sqrt{a_j}-\sqrt{a_i}).
\]
For $a_j > a_i$, the power-series expansion of the right hand side has nonnegative coefficients ($F$ has, by Lemma \ref{lm:mom}, and plainly so does $(1-t\sqrt{a_j})^{-1}$). Combining this with the Schur-Ostrowski criterion finishes the argument.
\end{proof}

\subsection{Proof of Theorem \ref{thm:M}}

We assume throughout that $a_1 \geq a_2 \geq \dots$. For the proofs, we recall several lemmas. The first one is classical and goes back to Moriguti.

\begin{lemma}[Moriguti, \cite{Mo}]\label{lm:Moriguti}
For every random variable $X$, $M(X) \geq \frac{1}{\sqrt{12}}\frac{1}{\sqrt{\Var(X)}}$.
\end{lemma}

{\red The second lemma is a reverse bound for log-concave random variables, that is having densities of the form $e^{-V}$ for a convex function $V\colon \R\to(-\infty,+\infty]$.

\begin{lemma}[Fradelizi \cite{Fr}, Bobkov-Chistyakov, \cite{BCh}]\label{lm:BCh}
For every log-concave random variable $X$, $M(X) \leq \frac{1}{\sqrt{\Var(X)}}$.
\end{lemma}
}

{\red The third lemma  is a straightforward extension of Lemma 3 from \cite{BNU}.
It relies on the Fourier inversion formula to derive bounds on densities, allowing to leverage independence. This technique has been particularly fruitful in geometric questions concerning sharp bounds on volumes of sections of $\ell_p$-balls, perhaps pioneered by Hensley in his paper \cite{Hen} on the cube, see also Ball's celebrated work \cite{Ball} as well as Koldobsky's works \cite{K98, Kol} for an in-depth general treatment, with the topic enjoying significant recent activity, see, e.g., \cite{CKT, CNT, KRZ, KK, KoKo, LPP, MR, MR2}.
}

\begin{lemma}\label{lm:ch-fun}
If $a_1 \leq \frac{1}{m}$ for a positive integer $m$, then the characteristic function $\phi$ of $\sqrt{a_1}X_1 + \dots + \sqrt{a_n}X_n$ satisfies
\begin{equation}\label{eq:ch-fun}
|\phi(t)| \leq (1+t^2/m)^{-m\gamma/2}, \qquad t \in \R.
\end{equation}
Moreover, if $m\gamma > 1$,
\begin{equation}\label{eq:density}
M\left(\sqrt{a_1}X_1 + \dots + \sqrt{a_n}X_n\right) \leq \frac{\sqrt{m}\Gamma\left(\frac{m\gamma - 1}{2}\right)}{2\sqrt{\pi}\Gamma\left(\frac{m\gamma}{2}\right)}.
\end{equation}
\end{lemma}
\begin{proof}
The characteristic function of $X_1$ is
\[
\E e^{itX_1} = (1-it)^{-\gamma}, \qquad t \in \R,
\]
with choosing, say the principal branch. Thus,
\[
\phi(t) = \prod_{j=1}^n (1-i\sqrt{a_j}t)^{-\gamma}
\]
and
\[
\log|\phi(t)| = -\frac{\gamma}{2}\sum_{j=1}^n \log(1+a_jt^2).
\]
To finish the proof of \eqref{eq:ch-fun}, we find the maximum of the convex function
\[
(a_1,\dots, a_n) \mapsto -\sum_{j=1}^n \log (1+a_jt^2)
\]
over the domain $D = \{(a_1, \dots, a_n), \ a_1, \dots, a_n \geq 0, \ a_1+\dots  +a_n = 1\} \cap [0,1/m]^n$. We can either follow \cite{BNU} verbatim and examine its extreme points, or, alternatively, it is clear that an arbitrary vector in $D$ is majorised by the vector $(\frac{1}{m},\dots, \frac{1}{m},0,\dots,0)$ (with $\frac{1}{m}$ repeated $m$-times and $0$ repeated $n-m$ times), so the lemma follows from the Schur-convexity of this function.

To see \eqref{eq:density}, we apply the Fourier inversion formula and \eqref{eq:ch-fun},
\begin{align*}
M\left(\sqrt{a_1}X_1 + \dots + \sqrt{a_n}X_n\right) \leq  \frac{1}{2\pi}\int_{\R} |\phi(t)|\dd t &\leq  \frac{1}{2\pi}\int_{\R} (1+t^2/m)^{-m\gamma/2}  \dd t \\
&= \frac{\sqrt{m}\Gamma\left(\frac{m\gamma - 1}{2}\right)}{2\sqrt{\pi}\Gamma\left(\frac{m\gamma}{2}\right)}.
\end{align*}
\end{proof}

We will also need a simple point-wise bound on the density of the sum $\sqrt{a_1}X_1+\dots+\sqrt{a_n}X_n$.

\begin{lemma}\label{lm:density}
The density $p$ of $\sqrt{a_1}X_1+\dots+\sqrt{a_n}X_n$ satisfies
\[
\frac{1}{\Gamma(n\gamma)}(a_1\dots a_n)^{-\gamma/2}x^{n\gamma - 1}e^{-x/\sqrt{a_n}} \leq p(x) \leq \frac{1}{\Gamma(n\gamma)}(a_1\dots a_n)^{-\gamma/2}x^{n\gamma - 1}.
\]
\end{lemma}
\begin{proof}
Fix $x > 0$. By independence, convolving the densities of $\sqrt{a_j}X_j$ yields
\begin{equation}\label{eq:pp}
\begin{split}
p(x) = &\Gamma(\gamma)^{-n}(a_1\dots a_n)^{-\gamma/2} \\
&\int_{\substack{t_1, \dots, t_{n-1} > 0, \\ t_1 + \dots + t_{n-1} < x}} \Bigg[ (t_1\dots t_{n-1})^{\gamma-1}(x-t_1-\dots-t_{n-1})^{\gamma-1} \\
&\hspace{2em}  \cdot \exp\left\{-\frac{t_1}{\sqrt{a_1}}-\dots-\frac{t_{n-1}}{\sqrt{a_{n-1}}} - \frac{x-t_1-\dots-t_{n-1}}{\sqrt{a_n}}\right\}\Bigg] \dd t_1\dots \dd t_{n-1}.
\end{split}
\end{equation}
Changing each $t_j$ to $xt_j$ gives
\begin{align*}
&p(x) = \Gamma(\gamma)^{-n}(a_1\dots a_n)^{-\gamma/2}x^{n\gamma-1} \\\notag
&\ \cdot\int_{\substack{t_1, \dots, t_{n-1} > 0, \\ t_1 + \dots + t_{n-1} < 1}} \Bigg[ (t_1\dots t_{n-1})^{\gamma-1}(1-t_1-\dots-t_{n-1})^{\gamma-1} \\\notag
&\hspace{1.5em}  \cdot \exp\left\{-x\left(\frac{t_1}{\sqrt{a_1}} + \dots + \frac{t_{n-1}}{\sqrt{a_{n-1}}} + \frac{1-t_1 - \dots - t_{n-1}}{\sqrt{a_n}}\right) \right\}\Bigg] \dd t_1\dots \dd t_{n-1}.
\end{align*}
Note that for the $t_j$ from the integral's domain,
\begin{align*}
0 &\leq \frac{t_1}{\sqrt{a_1}} + \dots + \frac{t_{n-1}}{\sqrt{a_{n-1}}} + \frac{1-t_1 - \dots - t_{n-1}}{\sqrt{a_n}} \\
&=\frac{1}{\sqrt{a_n}} + \sum_{j=1}^{n-1} t_j\left(\frac{1}{\sqrt{a_j}}-\frac{1}{\sqrt{a_n}}\right) \leq \frac{1}{\sqrt{a_n}}
\end{align*}
(recalling that $a_j \geq a_n$). The resulting estimates on $\exp\left\{\dots\right\}$ in the integrand give the desired bounds on $p$, where the factor $\frac{1}{\Gamma(n\gamma)}$ comes from
\begin{align*}
&\Gamma(\gamma)^{-n}\int_{\substack{t_1, \dots, t_{n-1} > 0, \\ t_1 + \dots + t_{n-1} < 1}}(t_1\dots t_{n-1})^{\gamma-1}(1-t_1-\dots-t_{n-1})^{\gamma-1} \dd t_1\dots \dd t_{n-1} \\
&= \Gamma(\gamma)^{-n}B(\gamma,\gamma)B(2\gamma,\gamma)\dots B((n-1)\gamma,\gamma) = \frac{1}{\Gamma(n\gamma)}.
\end{align*}
\end{proof}


We move to the proof of Theorem \ref{thm:M}. First we assume that $\gamma \geq 1$.

\begin{proof}[Proof of \eqref{eq:g1}, the lower bound.]
It immediately follows from Lemma \ref{lm:Moriguti} since we have, $\Var(\sum \sqrt{a_j}X_j) = \gamma$.
\end{proof}

\begin{proof}[Proof of \eqref{eq:g1}, the upper bound.]
{\red It immediately follows from Lemma \ref{lm:BCh} since for $\gamma \geq 1$, $\Gamma(\gamma)$ random variables are log-concave and sums of independent log-concave random variables are log-concave.}
\end{proof}

Now we assume that $\gamma < 1$. The upper bound in \eqref{eq:g12} as well as Remark \ref{rem:M} follow from the following upper bound.

\begin{theorem}\label{thm:upp-bd}
Fix a positive integer $k$ and let $\frac{1}{k+1} \leq \gamma < \frac{1}{k}$. Then
\begin{equation}\label{eq:gk}
M\left(\sum_{j=1}^n \sqrt{a_j}X_j\right) \leq C_\gamma(a_1\dots a_k)^{-\gamma/2}(1-a_1-\dots-a_k)^{\frac{k\gamma-1}{2}}
\end{equation}
with the right hand side understood as $+\infty$ when $n \leq k$. Constant $C_\gamma$ depends only on $\gamma$.
\end{theorem}

\begin{remark}\label{rem:smalln} 
In particular, when $k=1$, this gives the upper bound in \eqref{eq:g12}. Unfortunately, when $\gamma < \frac{1}{2}$, that is $k \geq 2$, bound \eqref{eq:gk} is not optimal: consider for instance the case when $a_1 = \dots = a_n = \frac{1}{n}$ with large $n$. However, note that when $n = k+1$, bound \eqref{eq:gk} is matched from below by Lemma \ref{lm:density} which gives that in this case
\[
M\left(\sum_{j=1}^{k+1} \sqrt{a_j}X_j\right) \geq c_\gamma(a_1\dots a_k)^{-\gamma/2}a_{k+1}^{\frac{k\gamma-1}{2}}
\]
with $c_\gamma = \frac{((k+1)\gamma-1)^{(k+1)\gamma-1}}{e^{(k+1)\gamma-1}\Gamma((k+1)\gamma)}$, justifying Remark \ref{rem:M}.
\end{remark}

\begin{proof}[Proof of \eqref{eq:gk}.]
In the course of the proof the value of $C_\gamma$ may change from line to line. Note that when $n \leq k$, by Lemma \ref{lm:density}, the maximum of the density is $+\infty$ (because the exponent at $x$ is negative). Thus, we can assume that $n \geq k+1$.

\emph{Case $n = k+1$.}
From \eqref{eq:pp}, after changing the variables (scaling each $t_i$ by $\sqrt{a_{k+1}}x$) we have,
\begin{align*}
&p(\sqrt{a_{k+1}}x) =  \Gamma(\gamma)^{-k-1}(a_1\dots a_{k})^{-\gamma/2}a_{k+1}^{\frac{k\gamma-1}{2}}x^{(k+1)\gamma-1} \\\notag
&\ \cdot\int_{\substack{t_1, \dots, t_{k} > 0, \\ t_1 + \dots + t_{k} < 1}} \Bigg[ (t_1\dots t_{k})^{\gamma-1}(1-t_1-\dots-t_{k})^{\gamma-1} \\\notag
&\hspace{1.5em}  \cdot \exp\left\{-x\left(\sqrt{\frac{a_{k+1}}{a_1}}t_1 + \dots + \sqrt{\frac{a_{k+1}}{a_{k}}}t_k +1-t_1 - \dots - t_{k}\right) \right\}\Bigg] \dd t_1\dots \dd t_{k}.
\end{align*}
The crude bound $\sum \sqrt{\frac{a_{k+1}}{a_j}}t_j \geq 0$ yields
\begin{align*}
&p(\sqrt{a_{k+1}}x) \leq  \Gamma(\gamma)^{-k-1}(a_1\dots a_{k})^{-\gamma/2}a_{k+1}^{\frac{k\gamma-1}{2}}x^{(k+1)\gamma-1} \\
&\ \cdot\int_{\substack{t_1, \dots, t_{k} > 0, \\ \sum t_j < 1}} \Bigg[ (t_1\dots t_{k})^{\gamma-1}\left(1-\sum t_j\right)^{\gamma-1}  \cdot \exp\left\{-x\left(1-\sum t_j\right) \right\}\Bigg] \dd t_1\dots \dd t_{k}.
\end{align*}
Using 
\[
x^{(k+1)\gamma-1}\exp\left\{-x\left(1-\sum t_j\right) \right\} \leq L_\gamma\left(1-\sum t_j\right)^{1-(k+1)\gamma},
\]
where $L_\gamma = \sup_{x > 0} x^{(k+1)\gamma-1}e^{-x} = ((k+1)\gamma-1)^{(k+1)\gamma-1}e^{-((k+1)\gamma-1)}$, we obtain the desired bound
\begin{equation}\label{eq:k+1}
\|p\|_\infty \leq C_\gamma(a_1\dots a_{k})^{-\gamma/2}a_{k+1}^{\frac{k\gamma-1}{2}}
\end{equation}
with
\[
C_\gamma = \Gamma(\gamma)^{-k-1}L_\gamma\int_{\substack{t_1, \dots, t_{k} > 0, \\ \sum t_j < 1}}  (t_1\dots t_{k})^{\gamma-1}\left(1-\sum t_j\right)^{-k\gamma} \dd t_1\dots \dd t_{k}
\]
which is finite because $k\gamma < 1$.

\emph{Case $n \geq k+2$.} If $a_1 \leq \frac{1}{k+2}$, then \eqref{eq:density} applied to $m = k+2$ gives
\[
M\left(\sum_{j=1}^n \sqrt{a_j}X_j\right) \leq C_\gamma \leq C_\gamma(a_1\dots a_k)^{-\gamma/2}\left(1-a_1-\dots-a_k\right)^{\frac{k\gamma-1}{2}},
\]
since $a_1\dots a_k \leq 1$, $1-a_1-\dots-a_k \leq 1$, where $C_\gamma$ only depends on $\gamma$. Now we assume that $a_1 > \frac{1}{k+2}$, write
\[
\sum_{j=1}^n \sqrt{a_j}X_j = \sqrt{\alpha}\eta + \sqrt{1-\alpha}\xi
\]
with
\[
\alpha = \sum_{j=1}^k a_j
\]
and
\[
\eta = \sum_{j=1}^k \sqrt{\frac{a_j}{\alpha}}X_j, \qquad \xi = \sum_{j=k+1}^n \sqrt{\frac{a_j}{1-\alpha}}X_j.
\]
We break the argument into two further cases depending on whether we can guarantee that $\xi$ has a bounded density (using Lemma \ref{lm:ch-fun}).

\emph{Case $\frac{a_{k+1}}{1-\alpha} \leq \frac{1}{k+2}$.} Here necessarily the number of summands in $\xi$ is at least $k+2$ (by comparing the largest coefficient to the average). Let $g$ be the density of $\xi$. By \eqref{eq:density} applied with $m=k+2$, we get $\|g\|_\infty \leq C_\gamma$. Moreover, if we let $f$ be the density of $\eta$, we know by Lemma \ref{lm:density} that
\[
f(x) \leq \frac{1}{\Gamma(k\gamma)}\alpha^{k\gamma/2}(a_1\dots a_k)^{-\gamma/2}x^{k\gamma-1}, \qquad x > 0.
\]
Thus for the density $p$ of $\sum_{j=1}^n \sqrt{a_j}X_j$, we obtain
\begin{align*}
p(x) &= \int_{0}^{x} \frac{1}{\sqrt{\alpha(1-\alpha)}}
f\left(\frac{t}{\sqrt{\alpha}}\right)g\left(\frac{x-t}{\sqrt{1-\alpha}}\right) \dd t \\
&\leq \frac{1}{\Gamma(k\gamma)\sqrt{1-\alpha}}(a_1\dots a_k)^{-\gamma/2}\int_0^x t^{k\gamma-1}g\left(\frac{x-t}{\sqrt{1-\alpha}}\right) \dd t.
\end{align*}
Changing $x$ to $\sqrt{1-\alpha}x$ and $t$ to $\sqrt{1-\alpha}t$, we get
\[
p(\sqrt{1-\alpha}x) \leq \frac{1}{\Gamma(k\gamma)}(a_1\dots a_k)^{-\gamma/2}(1-\alpha)^{\frac{k\gamma-1}{2}}\int_0^x t^{k\gamma-1}g\left(x-t\right) \dd t.
\]
It remains to observe that the resulting integral is bounded,
\begin{align*}
\int_0^x t^{k\gamma-1}g\left(x-t\right) \dd t \leq \|g\|_\infty\int_{t<1} t^{k\gamma-1}\dd t + \int_{t > 1} g(x-t)\dd t \leq \frac{\|g\|_\infty}{k\gamma} + 1.
\end{align*}

\emph{Case $\frac{a_{k+1}}{1-\alpha} \geq \frac{1}{k+2}$.} Plainly,
\[
M\left(\sum_{j=1}^n \sqrt{a_j}X_j\right) \leq M\left(\sum_{j=1}^{k+1} \sqrt{a_j}X_j\right) = \frac{1}{\sqrt{A}}M\left(\sum_{j=1}^{k+1} \sqrt{\frac{a_j}{A}}X_j\right)
\]
with $A = \sum_{j=1}^{k+1} a_j$. By the case $n=k+1$, i.e. \eqref{eq:k+1},
\[
M\left(\sum_{j=1}^{k+1} \sqrt{\frac{a_j}{A}}X_j\right) \leq C_\gamma(a_1\dots a_k)^{-\gamma/2}A^{k\gamma/2}(a_{k+1}/A)^{\frac{k\gamma-1}{2}},
\]
thus
\begin{align*}
M\left(\sum_{j=1}^n \sqrt{a_j}X_j\right) &\leq C_\gamma(a_1\dots a_k)^{-\gamma/2}a_{k+1}^{\frac{k\gamma-1}{2}} \\
&\leq C_\gamma(k+2)^{-\frac{k\gamma-1}{2}}(a_1\dots a_k)^{-\gamma/2}(1-\alpha)^{\frac{k\gamma-1}{2}},
\end{align*}
as desired. This concludes the proof of \eqref{eq:gk}.
\end{proof}

\begin{proof}[Proof of \eqref{eq:g12}, the lower bound]
We assume that $0 < \gamma < 1$ and denote $Z = \sum \sqrt{a_j}X_j$. The argument from \cite{BNU} can be repeated almost verbatim. We include it for completeness.  

\emph{Case $a_1 \leq \frac12$}. Since $\Var(Z) =\gamma$, Lemma \ref{lm:Moriguti} yields $M(Z)\gr \frac{1}{2\sqrt{3\gamma}}$ so if $a_1\ls 1/2$ then $M(Z)\gr c_\gamma(1-a_1)^{\frac{\gamma-1}{2}}$ with $c_\gamma=(2^{\frac{3-\gamma}{2}}\sqrt{3\gamma})^{-1}$.

\emph{Case $a_1\gr 1/2$}. Let $\xi=\sum_{j=2}^n \frac{\sqrt{a_j}}{\sqrt{1-a_1}}X_j$, so that $Z=\sqrt{a_1}X_1+\sqrt{1-a_1}\xi$. Note that $\xi$ is independent of $\sqrt{a_1}X_1$ so the density $f_Z$ of $Z$ is given by the convolution of the densities $f_{\sqrt{a_1}X_1}$ and $f_{\sqrt{1-a_1}\xi}$. We have,
\[
f_Z(x) = \frac{1}{\Gamma(\gamma)\sqrt{a_1(1-a_1)}}\int_0^{x} (\frac{x-t}{\sqrt{a_1}})^{\gamma-1}\exp\left(-\frac{x-t}{\sqrt{a_1}}\right)f_\xi\left(\frac{t}{\sqrt{1-a_1}}\right)\,dt,
\] 
and, applying this for $x\sqrt{1-a_1}$,
\[
f_Z(x\sqrt{1-a_1}) = \frac{(1-a_1)^\frac{\gamma-1}{2}}{\Gamma(\gamma)a_1^{\gamma/2}}\int_0^x (x-t)^{\gamma-1}\exp\left(-\frac{\sqrt{1-a_1}}{\sqrt{a_1}}(x-t)\right)f_\xi(t)\,dt.
\] 
We will use this identity for $x=\E\xi+2$, lower bounding the expression in the right hand side by integrating on the interval $I=(\max(\E\xi-2,0),\E\xi+2)$. Note that $x-t\ls 4$ for every $t\in I$. It follows that
\[
M(Z) \gr \frac{(1-a_1)^\frac{\gamma-1}{2}}{\Gamma(\gamma)a_1^{\gamma/2}} 4^{\gamma-1}\exp\left(-\frac{4\sqrt{1-a_1}}{\sqrt{a_1}}\right)\cdot \pp(\xi\in I).
\]
The assumption $a_1\gr 1/2$ yields $\frac{1-a_1}{a_1}\ls 1$. Since $\Var(\xi)=\gamma$, we get by Chebyshev's inequality that
\[
\pp(\xi\in I) = 1-\pp(|\xi-\E\xi|\gr 2) \gr 1-\frac{1}{4}\Var(\xi) = 1-\frac{\gamma}{4}.
\]
Putting these together and the trivial bound $a_1^{\gamma/2} \leq 1$, we get the lower bound $M(Z) \gr c_\gamma(1-a_1)^\frac{\gamma-1}{2}$ with $c_\gamma=\frac{4^{\gamma}(4-\gamma)}{4^2e^4\Gamma(\gamma)}$.

Combining the two cases together, the lower bound in \eqref{eq:g12} holds with
\[
c_\gamma = \min\left\{(2^{\frac{3-\gamma}{2}}\sqrt{3\gamma})^{-1}, \frac{4^{\gamma}(4-\gamma)}{4^2e^4\Gamma(\gamma)} \right\} \geq \min\left\{\frac{1}{2^{3/2}\sqrt{3}}, \frac{3\gamma}{4^2e^4}\right\}  > 0.003\gamma,
\]
since $\frac{1}{\Gamma(\gamma)} \geq \gamma$.
\end{proof}



\begin{thebibliography}{9}


\bibitem{Ball}
Ball, K.,
Cube slicing in $R^n$. \emph{Proc. Amer. Math. Soc.} 97 (1986), no. 3, 465--473.


\bibitem{BNZ}
Bartczak, M., Nayar, P., Zwara, S.,
Sharp variance-entropy comparison for nonnegative Gaussian quadratic forms, \emph{IEEE Trans. Inform. Theory} 67 (2021), no. 12, 7740--7751.


\bibitem{Bh}
Bhatia, R.,
Matrix analysis. 
Graduate Texts in Mathematics, 169. \emph{Springer-Verlag, New York}, 1997.

\bibitem{BN}
Bia\l obrzeski, M., Nayar, P., R\'enyi entropy and variance comparison for symmetric log-concave random variables, preprint (2021), arXiv:2108.10100.

\bibitem{BCh}
Bobkov, S. G., Chistyakov, G. P.,
On concentration functions of random variables. \emph{J. Theoret. Probab.} 28 (2015), no. 3, 976--988.

\bibitem{BCh2}
Bobkov, S. G., Chistyakov, G. P., 
Entropy power inequality for the R\'enyi entropy. 
\emph{IEEE Trans. Inform. Theory} 61 (2015), no. 2, 708--714.

\bibitem{BMM}
Bobkov, S. G., Marsiglieti, A., Melbourne, J., 
Concentration functions and entropy bounds for discrete log-concave distributions. \emph{Combinatorics, Probability and Computing} 31 (2022), 54--72.

\bibitem{BM}
Bobkov, S. G., Marsiglietti. A.,
Variants of the entropy power inequality. \emph{IEEE Trans. Inform. Theory}, 63(12):7747--7752, 2017.

\bibitem{BNU}
Bobkov, S. G., Naumov, A. A., Ulyanov, V. V.,
Two-sided inequalities for the density function's maximum of weighted sum of chi-square variables, preprint (2020) arXiv:2012.10747. In: \textit{Recent Developments in Stochastic Methods and Applications}. Springer Proceedings in Mathematics \& Statistics, 371 (2021), pp. 178--189. Springer, Cham.


\bibitem{CGT}
Chasapis, G., Gurushankar, K., Tkocz, T.,
Sharp bounds on $p$-norms for sums of independent uniform random variables, $0<p<1$, preprint (2021), arXiv:2105.14079. To appear in \emph{J. Anal. Math.}


\bibitem{CKT}
G. Chasapis, H. K\"onig, T. Tkocz,
From Ball's cube slicing inequality to Khinchin-type inequalities for negative moments, J. Funct. Anal. 281 (2021), no. 9, Paper No. 109185, 23 pp.


\bibitem{CNT}
Chasapis, G., Nayar, P., Tkocz, T.,  Slicing $\ell_p$-balls reloaded: stability, planar sections in $\ell_1$, preprint (2021), arXiv:2109.05645, to appear in \emph{Ann. Probab.}


\bibitem{ENT1}
Eskenazis, A., Nayar, P., Tkocz, T.,
Gaussian mixtures: entropy and geometric inequalities, \emph{Ann. of Prob.} 46(5) 2018, 2908--2945.


\bibitem{ENT2}
Eskenazis, A., Nayar, P., Tkocz, T.,
Sharp comparison of moments and the log-concave moment problem.
\emph{Adv. Math.} 334 (2018), 389--416.


\bibitem{Fel}
Feller, W., An introduction to probability theory and its applications. Vol. II. Second edition \emph{John Wiley \& Sons, Inc., New York-London-Sydney} 1971.


\bibitem{Fr}
Fradelizi, M., Hyperplane sections of convex bodies in isotropic position. \emph{Beitr\"age Algebra Geom.}, 40(1) (1999), 163--183.


\bibitem{Hen}
Hensley, D.
Slicing the cube in $R^n$ and probability (bounds for the measure of a central cube slice in $R^n$ by probability methods).
\emph{Proc. Amer. Math. Soc.} 73 (1979), no. 1, 95--100.




\bibitem{K98} 
Koldobsky, A., An application of the Fourier transform to sections of star bodies, \emph{Israel J. Math.} 106 (1998), 157--164.

\bibitem{Kol}
Koldobsky, A., Fourier analysis in convex geometry. Mathematical Surveys and Monographs, 116. \emph{American Mathematical Society, Providence, RI}, 2005.


\bibitem{KRZ}
Koldobsky, A., Ryabogin, D., Zvavitch, A.,
Fourier analytic methods in the study of projections and sections of convex bodies. \emph{Fourier analysis and convexity}, 119--130,
Appl. Numer. Harmon. Anal., \emph{Birkh\"auser Boston, Boston, MA}, 2004.


\bibitem{KK}
K\"onig, H., Koldobsky, A.,
On the maximal measure of sections of the $n$-cube. \emph{Geometric analysis, mathematical relativity, and nonlinear partial differential equations}, 123--155,
Contemp. Math., 599, \emph{Amer. Math. Soc., Providence, RI}, 2013.


\bibitem{KoKo}
K\"onig, H., Koldobsky, A., On the maximal perimeter of sections of the cube. \emph{Adv. Math.} 346 (2019), 773--804. 

\bibitem{Li}
Li, J., R\'enyi entropy power inequality and a reverse. 
\emph{Studia Math.}, 242(3):303--319, 2018.

\bibitem{LMM}
Li, J., Marsiglietti, A., Melbourne, J., Further investigations of R\'enyi entropy power inequalities and an entropic characterization of s-concave densities. \emph{Geometric Aspects of Functional Analysis Israel Seminar (GAFA)} 2017-2019, (Lecture Notes in Mathematics 2262), 2020.

\bibitem{LPP}
Livshyts, G., Paouris, G., Pivovarov, P.,
On sharp bounds for marginal densities of product measures.
\emph{Israel J. Math.} 216 (2016), no. 2, 877--889.


\bibitem{MMX}
Madiman, M., Melbourne, J,. Xu, P., Forward and reverse entropy power inequalities in convex geometry. \emph{Convexity and Concentration}, 427--485, 2017.

\bibitem{MNT}
Madiman, M., Nayar, P., Tkocz, T., Sharp moment-entropy inequalities for symmetric log-concave distributions, \emph{IEEE Trans. Inform. Theory} 67 (2021), no. 1, 81--94.

\bibitem{MM}
Marsiglietti, A., Melbourne, J., On the entropy power inequality for the R\'enyi entropy of order $[0, 1]$. \emph{IEEE Trans. Inform. Theory}, 65(3):1387--1396, 2019.

\bibitem{MR}
Melbourne, J., C. Roberto, C.,
Quantitative form of Ball's Cube slicing in Rn and equality cases in the min-entropy power inequality, preprint (2021), arXiv:2109.03946.


\bibitem{MR2}
Melbourne, J., C. Roberto, C.,
Transport-majorization to analytic and geometric inequalities, preprint (2021), arXiv:2110.03641.

\bibitem{MT}
Melbourne, J., Tkocz, T., Reversals of R\'enyi entropy inequalities under log-concavity. \emph{IEEE Trans. Inform. Theory} 67 (2021), no. 1, 45--51.


\bibitem{Mo}
Moriguti, S.,
A lower bound for a probability moment of any absolutely continuous distribution with finite variance.
\emph{Ann. Math. Statistics} 23 (1952), 286--289.



\bibitem{HV}
Nguyen, H. H., Vu, V. H., Small ball probability, inverse theorems, and applications. \emph{Erd\"os centennial}, 409--463, Bolyai Soc. Math. Stud., 25, \emph{J\'anos Bolyai Math. Soc.}, Budapest, 2013.

\bibitem{RS}
Ram E., Sason, I., On R\'enyi entropy power inequalities. 
\emph{IEEE Trans. Inform. Theory}, 62(12):6800--6815, 2016.

\bibitem{R}
R\'enyi, A., On measures of entropy and information. In \emph{Proc. 4th Berkeley Sympos. Math. Statist. and Prob., Vol. I}, 547--561. \emph{Univ. California Press, Berkeley, Calif.}, 1961.

\bibitem{Yu}
Yu, Y.,
On an inequality of Karlin and Rinott concerning weighted sums of i.i.d.
random variables. \emph{Advances in Applied Probability}, 40(4):1223--1226, 2008.


\end{thebibliography}
\end{document}